\documentclass[english,12pt]{amsart}

\usepackage{multicol} 
\usepackage{babel}
\usepackage{amstext}
\usepackage{amsmath}
\usepackage{amsfonts}
\usepackage{latexsym}
\usepackage{ifthen}
\usepackage{xypic}
\usepackage{amssymb}

\xyoption{all}
\newcommand{\Rien}[1]{}
\newcommand{\Formel}[1]{(\ref{#1})}

\newcommand{\order}{\mathcal O}
\newcommand{\Sieg}{\mathfrak H}

\newcommand{\B}{\mathbb B}

\newcommand{\id}{{\rm id}}
\renewcommand{\O}{{\mathcal O}}

\newcommand{\PN}{{\mathbb P}}

\newcommand{\lra}{\longrightarrow}
\newcommand{\KC}{{\mathbb C}}
\newcommand{\KZ}{{\mathbb Z}}
\newcommand{\KQ}{{\mathbb Q}}

\newcommand{\KN}{{\mathbb N}}
\newcommand{\End}{{\rm End}}

\newcommand{\KR}{\mathbb R}

\newtheorem{lemma1}[equation]{}

\newenvironment{lemma}{\begin{lemma1}{\bf Lemma.}}{\end{lemma1}}

\newenvironment{example}{\begin{lemma1}{\bf Example.}\rm}{\end{lemma1}}

\newenvironment{theorem}{\begin{lemma1}{\bf Theorem.}}{\end{lemma1}}

\newenvironment{proposition}{\begin{lemma1}{\bf Proposition.}}{\end{lemma1}}

\newenvironment{rem}{\begin{lemma1}{\bf Remark.}\rm}{\end{lemma1}}

\begin{document}

\title{Splitting submanifolds of families \\ of fake elliptic curves}
\author[P. Jahnke]{Priska Jahnke}
\address{P. Jahnke - Fachbereich Mathematik und Informatik - Freie Universit\"at Berlin - Arnimallee 3 - D-14195 Berlin, Germany}
\email{priska.jahnke@fu-berlin.de}
\author[I. Radloff]{Ivo Radloff}
\address{I. Radloff - Mathematisches Institut - Universit\"at Bayreuth
  - D-95440 Bayreuth, Germany}
\email{ivo.radloff@uni-bayreuth.de}
\date{\today}
\maketitle

\section*{Introduction}

Let $M_m$ be a compact complex manifold and $N_n$ a compact complex submanifold, $0 < n < m$. We say $N$ {\em splits in $M$} if the holomorphic tangent bundle sequence 
  \begin{equation} \label{TS}
    0 \lra T_N \lra T_M|_N \lra N_{N/M} \lra 0
  \end{equation}
splits holomorphically. Recall that this is the case if and only $id_{N_{N/M}}$ is mapped to zero under $H^0(N, N_{N/M}^* \otimes N_{N/M}) \lra H^1(N, N^*_{N/M} \otimes T_N)$. 

The splitting of \Formel{TS} is not particularly geometric in certain cases (\cite{J}. The situation is different if one imposes strong conditions on $M$. The following result is due to Mok:

\begin{proposition} \label{MokP} (\cite{Mok})
  Let $M$ be compact K\"ahler--Einstein with constant holomorphic sectional curvature. Then a submanifold $N$ splits in $M$ if and only if $N$ is totally geodesic.
\end{proposition}

The manifolds $M$ in question here are $M = \PN_m(\KC)$, finite {\'e}tale quotients of complex tori and ball quotients, i.e., manifolds whose universal covering space $\tilde{M}$ is $\B_m(\KC) = \{z \in \KC^m | \; |z| < 1\}$, the non--compact dual of $\PN_m(\KC)$ in the sense of hermitian symmetric spaces. The compact submanifolds $N$ that split in $M$ can be described quite explicitely: in the case of $M = \PN_m(\KC)$, $N$ splits if and only if $N$ is a linear subspace (\cite{VdV}); if $M$ is a torus, then $N$ splits if and only if $N$ is a subtorus (\cite{J}). For the case $M$ a ball quotient see \cite{Ye}. 

The three types of $M$ share the common property that the universal covering space $\tilde{M}$ can be embedded into $\PN_m(\KC)$ such that $\pi_1(M)$ acts as a subgroup of $PGl_m(\KC)$. Manifolds with this property carry a holomorphic projective structure in the sense of Gunning (\cite{Gunning}). This property is important in Mok's proof.

In the case $M$ a compact Riemann surface or $M$ a projective complex surface (\cite{KO}), the existence of such a structure already implies $M$ K\"ahler Einstein and of constant holomorphic sectional curvature. In three dimensions this is no longer true (\cite{JRproj}): 

\begin{theorem}
  A projective threefold $M$ carries a holomorphic projective connection if and only if up to finite {\'e}tale coverings
\begin{enumerate}
 \item $M \simeq \PN_m(\KC)$ or
 \item $M$ is an abelian threefold or
 \item $M$ is a modular family of fake elliptic curves or
 \item $M$ is ball quotient.
\end{enumerate}
In any case the connection is flat.
\end{theorem}
The above case 3.) is not covered by Mok's result. A modular family of fake elliptic curves for our purposes is a projective threefold $M$ that admits a holomorphic submersion 
  \[\pi: M \lra C\]
onto a compact Shimura curve $C$ such that every fiber is a smooth abelian surface $A$ and such that for the general fiber ${\rm End}_{\KQ}(A)$ is an indefinite division quaternion algebra over $\KQ$. We recall the construction in detail in \S 3. What we prove here is the following result which completes the classification in the projective case up to dimension three:

\begin{proposition} \label{result}
  Let $\pi: M \lra C$ be a modular family of fake elliptic curves. A compact submanifold $N$ of $M$ of dimension $0 < \dim N < 3$ splits in $M$ if and only if
\begin{enumerate}
  \item $N$ is an {\'e}tale multisection of $\pi$ or
  \item $N$ is an elliptic curve in a fiber of $\pi$.
\end{enumerate}
\end{proposition}

\

\noindent {\bf Notations.} Manifolds are complex manifolds, $\Omega_X^1$ denotes the bundle of holomorphic $1$--forms, $T_X$ the holomorphic tangent bundle. We do not distinguish between Cartier divisors and line bundles, e.g., $K_X = \omega_X = \det\Omega_X^1$. 

For a vector bundle $E$, $\PN(E)$ denotes the ``hyperplane'' bundle (defined by projectivizations of the transition functions of $E^*$).

\section{Holomorphic normal projective connections}
\setcounter{equation}{0}

There are definitions of projective structures and connections in the language of principal bundles. We essentially follow Kobayashi and Ochiai (\cite{KO}).

We first recall the notion of the {\em Atiyah class} (\cite{At}) associated to a holomorphic vector bundle $E$ on the complex manifold $M$: It is the splitting class $b(E) \in H^1(M, Hom(E,E)\otimes \Omega_M^1)$ of the first jet sequence
  \begin{equation} \label{jet}
     0 \lra \Omega_M^1 \otimes E \lra J_1(E) \lra E \lra 0,
  \end{equation}
i.e., the image of $id_E$ under the first connecting homomorphism of \Formel{jet} tensorized with $E^*$. 

If $\Theta^{1,1}$ denotes the $(1,1)$--part of the curvature tensor of some differentiable connection on $E$, then, under the Dolbeault isomorphism, $b(E)$ corresponds to $[\Theta^{1,1}] \in H^{1,1}(M, Hom(E,E))$. In particular, for $M$ K\"ahler
 \begin{equation} \label{fcc}
   trace(b(E)) = -2i\pi c_1(E) \in H^1(M, \Omega_M^1).
 \end{equation}
Because of \Formel{fcc} we consider $a(E) := -\frac{1}{2i\pi}b(E)$, the {\em normalized Atiyah class}. For properties and functiorial behaviour of these classes see \cite{At}.

% trace(rA) = r trace(A)

\subsection{Definitions} Let $M$ be some $m$--dimensional compact complex
K\"ahler manifold. Then $M$ carries a holomorphic normal projective connection if the normalised Atiyah class of the holomorphic cotangent bundle
has the form
  \begin{equation} \label{AtProj} 
a(\Omega_M^1) = \frac{c_1(K_M)}{m+1} \otimes id_{\Omega_M^1} + id_{\Omega_M^1}
  \otimes \frac{c_1(K_M)}{m+1} \in H^1(M, \Omega_M^1 \otimes T_M \otimes
  \Omega_M^1)
  \end{equation}
where we use the identities $\Omega_M^1 \otimes T_M \otimes
\Omega_M^1 \simeq
\End(\Omega_M^1) \otimes \Omega_M^1 \simeq \Omega_M^1 \otimes
\End(\Omega_M^1)$.

A Chech-solution to \Formel{AtProj} can be interpreted as a connection on $T_M$ satisfying certain conditions similar to the Schwarzian derivative (\cite{MM}). We will not make use of this fact.

$M$ is said to carry a projective structure if there exists a
holomorphic projective atlas, i.e., an atlas whose charts can be embedded into $\PN_m(\KC)$
such that the coordinate change is given by restrictions of
projective automorphisms. A manifold with a projective structure carries a (flat) projective connection, meaning that zero is a cocycle solution to \Formel{AtProj}.

\begin{example} \label{Exmpl} Compact complex manifolds that admit a flat holomorphic projetive connection:
 
  1.) $M = \PN_m(\KC)$. 

2.) Any manifold $M$ whose universal covering space $\tilde{M}$ can be embedden into $\PN_m(\KC)$ such that $\pi_1(M)$ acts by
restrictions of projective transformations admits a projective
structure. In particular

2.1.) Ball quotients $\B_m(\KC)/\Gamma$, where $\Gamma \subset SU(1,m)$ is discrete and torsion free and

2.2.) tori $\KC^m/\Lambda$ where $\Lambda \simeq \KZ^{2m}$, carry a projective structure (\cite{KO}). Note that 1) and 2) covers all compact K\"ahler Einstein manifolds with constant holomorphic sectional curvature.

3.) If $M$ carries a holomorphic projective connection and $M' \lra M$ is finite {\'e}tale, then $M'$ carries a holomorphic projective connection.
\end{example}

\section{Holomorphic projective connections and splitting submanifolds}
A general remark on coverings: Let $N$ (resp. $N'$) be a compact submanifold of some compact manifold $M$ (resp. $M'$). Let $\nu: M' \lra M$ be finite {\'e}tale such that $\nu(N') = N$. Then $N$ splits in $M$ if and only of $N'$ splits in $M'$. 

Indeed, the tangent bundle sequence \Formel{TS} of $N'$ in $M'$ is the pull back of the tangent bundle sequence of $N$ in $M$. By the trace map, $\O_M$ is a direct summand of $\nu_*\O_{M'}$. We have an inclusion in cohomology
  \[H^1(N, N_{N/M} \otimes T_N^*) \hookrightarrow H^1(N', N_{N'/M'} \otimes T_{N'}^*)\]
coming from Leray spectral sequence. The splitting class of \Formel{TS} of $N$ in $M$ is mapped to the splitting class of \Formel{TS} of $N'$ in $M'$. In this sense, we can always replace a pair $(M, N)$ by $(M', N')$.

\

We give a necessary but not sufficient condition for a submanifold to be split (see example~\ref{Exnonspl}):

\begin{proposition} \label{split}
  Let $M_m$ be compact K\"ahler carrying a holomorphic projective connection. Let $\iota: N_n \hookrightarrow M_m$ be a compact submanifold that splits in $M_m$. Then:
\begin{enumerate}
  \item $N$ carries a holomorphic projective connection.
  \item  $\frac{c_1(\iota^*K_M)}{m+1} = \frac{c_1(K_N)}{n+1}$ in $H^1(N, \Omega_N^1)$.
   \item $a(N^*_{N/M}) = id_{N_{N/M}} \otimes \frac{c_1(K_N)}{n+1} \in H^1(M, N_{N/M}^*\otimes N_{N/M} \otimes \Omega_N^1)$.
\end{enumerate}
\end{proposition}

\begin{proof} 
Sequence \Formel{jet} for a direct sum $E_1 \oplus E_2$ of vector bundles is the direct sum of the sequences associated to $E_1$ and $E_2$, respectively (\cite{At}, proposition~8). In particular $a(E_1 \oplus E_2)$ is the direct sum of $a(E_1)$ and $a(E_2)$ in a natural way. 

By assumption, $\iota^*\Omega^1_M \simeq \Omega^1_N\oplus N^*_{N/M}$. Then we can compute $a(\Omega^1_N)$ and $a(N_{N/M})$ from $a(\iota^*\Omega^1_M)$ by projecting the class to the corresponding summands.

The class $a(\Omega_M^1)$ is given by \Formel{AtProj} and hence
  \begin{equation} \label{AtPB}
a(\iota^*\Omega_M^1) = \frac{\iota^*c_1(K_M)}{m+1} \otimes d\iota + id_{\iota^*\Omega_M^1} \otimes \frac{c_1(\iota^*(K_M))}{m+1}
  \end{equation}
in $H^1(N, \iota^*\Omega_M^1 \otimes \iota^*T_M \otimes \Omega_N^1)$, where $d\iota: \iota^*\Omega_M^1 \lra \Omega_N^1$. Note that we distinguish between pull back of cohomology classes and pull back of forms. Now denote the splitting map by $s:\iota^*T_M \lra T_N$. 

The class $a(\Omega_N^1)$ is the image of $a(\iota^*\Omega_M^1)$ under the map induced by $d\iota \otimes s \otimes id$:
 \[H^1(N, \iota^*\Omega_M^1 \otimes \iota^*T_M \otimes \Omega_N^1) \lra H^1(\Omega_N^1 \otimes T_N \otimes \Omega_N^1).\]
Applying $d\iota$ to the first factor maps $a(\iota^*\Omega_M^1)$ to the following class in $H^1(M, \Omega_N^1 \otimes \iota^*T_M \otimes \Omega_N^1)$:
 \[\frac{c_1(\iota^*K_M)}{m+1} \otimes d\iota + d\iota \otimes \frac{c_1(\iota^*(K_M)}{m+1}.\]
The map $s$ induces a map $H^0(N, \iota^*T_M \otimes \Omega_N^1) \lra H^0(N, T_N\otimes \Omega_N^1)$, mapping $d\iota$ to $id_{\Omega^1_N}$ by definition. We therefore obtain
 \begin{equation} \label{AtN} 
   a(\Omega_N^1) = \frac{c_1(\iota^*K_M)}{m+1} \otimes id_{\Omega_N^1} + id_{\Omega_N^1} \otimes \frac{c_1(\iota^*K_M)}{m+1}.
 \end{equation}
The trace obtained by contracting the first two factors $\Omega_N^1 \otimes T_N$ yields $c_1(K_N) \in H^1(N, \Omega_N^1)$. Then \Formel{AtN} shows
  \[trace(a(\Omega^1_N) = \frac{c_1(\iota^*K_M)}{m+1} + n\frac{c_1(\iota^*K_M)}{m+1} = \frac{n+1}{m+1}c_1(\iota^*K_M)\]
and we obtain 2.). By \Formel{AtN} $N$ carries a holomorphic projective connection. Formula 3.) is obtained in the same way using
 \[H^1(N, \iota^*\Omega_M^1 \otimes \iota^*T_M \otimes \Omega_N^1) \lra H^1(N, N_{N/M}^* \otimes N_{N/M} \otimes \Omega_N^1).\]
We only remark that $\frac{\iota^*c_1(K_M)}{m+1} \otimes d\iota$ is mapped to zero.
\end{proof}

\begin{example} \label{Exspl}
  The K\"ahler--Einstein case:

  1.) By a result of Van de Ven (\cite{VdV}), a compact complex submanifold $N_n$ of $\PN_m(\KC)$ splits if and only if $N$ is a linear subspace. Here we have $\O_{\PN_n}(1) = \O_{\PN_m}(1)|_N$ and $N_{N/M} \simeq \O_{\PN_n}(1)^{\oplus m-n}$.

In the case of tori $\KC^m/\Gamma$ and ball quotients $\B_m(\KC)/\Gamma$ a splitting submanifold may or may not exist depending on the choice of $\Gamma$:

2.1.) By a result of one of the authors (\cite{J}), a compact complex submanifold $N_n$ of a torus $M = \KC^m/\Lambda$ splits if and only if $N$ is a subtorus. Here $\O_N = \O_T|_N$ and the normal bundle is trivial.

2.2.) See \cite{Ye} for an example in the case $\B_2(\KC)/\Gamma$.
\end{example}

\begin{example} \label{Exnonspl}
The three conditions in proposition~\ref{split} are not sufficient for a submanifold to split: Let $\Gamma$ be some torsion free congruence subgroup of $Sl_2(\KZ)$. Then $\Gamma$ acts without fixed points on $\Sieg_1$ as a group of Moebius transformations. The set of matrices
  \[\gamma_{m,n} = \left(\begin{array}{ccc}
      1 & m & n \\
      0 & a & b \\
      0 & c & d
    \end{array}\right), \quad \gamma = \left(\begin{array}{cc}
      a & b \\
      c & d
    \end{array}\right) \in \Gamma,\]
is a subgroup $\Gamma_{\Lambda}$ of $Sl_3(\KC)$ that acts projectively on $\KC \times \Sieg_1$, i.e., $\gamma_{m,n}(z, \tau)=(\frac{z+m\tau + n}{c\tau + n}, \frac{a\tau + b}{c\tau + d})$. The action is free, $U = \KC \times \Sieg_1/\Gamma_{\Lambda}$ is a smooth (non--compact) manifold. The canonical map
  \[\pi: U \lra \Sieg_1/\Gamma\]
is proper holomorphic with elliptic fibers $\pi^{-1}(\tau) \simeq \KC/(\KZ\tau + \KZ) =: E_{\tau}$. The map has sections. For the standard groups $\Gamma = \Gamma_0(n), \Gamma_1(n)$, we obtain the usual 'universal' families of elliptic curves.

By construction, $U$ has a projective structure. Any fiber $E_\tau$ satisfies the three conditions of proposition~\ref{split}. Nevertheless $E_{\tau}$ does not split in $U$ (This is well known or can be proved along the arguments of lemma~\ref{fibnonspl}).
\end{example}

\section{Modular families of fake elliptic curves}
\setcounter{equation}{0}

In this section we recall the construction and basic properties of families of fake elliptic curves following Shimura (\cite{Shim}). The construction depends an the choice of a PEL datum.

\subsection{Quaternions} Let $B$ be an indefinite division quaternion algebra over $\KQ$. Then $B \otimes_{\KQ} \KR \simeq M_{2\times 2}(\KR)$. Fix once and for all an embedding
  \[B \hookrightarrow M_{2\times 2}(\KR)\]
such that the reduced norm and trace are given by usual determinant and trace, respectively. From now on think of $B$ as a matrix group generated over $\KQ$ by
  \[x = \left(\begin{array}{cc}
              \sqrt{a} & 0 \\
                 0 & -\sqrt{a}
              \end{array}\right) \quad \mbox{and} \quad y = \left(\begin{array}{cc}
                                                    0 & b \\
                                                    1 & 0
                                                 \end{array}\right)\]
where $a,b \in \KQ$, $a > 0, b < 0$. We have $x^2 = a id = a, y^2 = b, xy=-yx$ and $B = \KQ + \KQ x + \KQ y + \KQ xy$. The usual quaternion (anti-) involution $'$ on $B$ given by $(k + lx + my + nxy)' = k - lx - my - nxy$. It extends to an involution on $M_{2 \times 2}(\KR)$ which we also denote by $'$.

\subsection{The abelian surface $A_{\tau}$} Any $\tau \in \Sieg_1$, the upper half plane of $\KC$, endows $\KR^4 \simeq M_{2\times 2}(\KR)$ with a complex structure
  \begin{equation} \label{CS}
    M_{2\times 2}(\KR) \lra \KC^2, \quad m \mapsto m_{\tau} := m\left(\begin{array}{c}\tau \\ 1\end{array}\right).
  \end{equation}
For the construction we may take a maximal order $\order_B \subset B$. Define 
  \begin{equation} \label{FEC} 
    \order_{B, \tau} := \order_B\left(\begin{array}{c}\tau \\ 1\end{array}\right), \quad A_{\tau} := \KC^2/\order_{B, \tau}.
  \end{equation}
Then $A_{\tau}$ is an abelian surface and there is a natural inclusion map
  \[\order_B \hookrightarrow {\rm End}(A_\tau)\]
given by multiplication from the left. The induced embedding $B \hookrightarrow End_{\KQ}(A_\tau)$ is an isomorphism iff $A_\tau$ is simple (see for example the classification of possible $End_{\KQ}$'s in \cite{BiLa}).

Projectivity can be seen as follows: Choose $\rho \in B$ such that $\rho = -\rho'$ and $\rho^2 < 0$. Then 
  \[E(m_1, m_2) := trace(\rho m_1 m_2')\]
is a symplectic, non degenerate form on $M_{2\times 2}(\KR)$. Some rational multiple of $E$ takes integral values on $\order_B$ and satisfies the Riemann conditions for any $\tau$, defining a polarization $H_{\tau}$ on $A_{\tau}$.

\subsection{Isomorphic $A_{\tau}$'s} Let $\order_{B, +}^* \subset \order_B$ be the group of units of positive norm. Then $\order_{B, +}^* \subset Sl_2(\KR)$. The group therefore acts on $\Sieg_1$ as a group of Moebius transformations.

If $\tau' = \gamma(\tau)$ for some $\gamma = \left(\begin{array}{cc} a & b \\
c & d \end{array}\right) \in \order_{B, +}^*$, then
  \[\order_{B, \tau'} = \order_B \left(\begin{array}{c}\gamma(\tau) \\ 1\end{array}\right) = \frac{1}{c\tau + d}\order_B \gamma \left(\begin{array}{c}\tau \\ 1\end{array}\right) =  \frac{1}{c\tau + d}\order_B \left(\begin{array}{c}\tau \\ 1\end{array}\right)\]
Multiplication by $\frac{1}{c\tau+d}$ therefore induces an isomorphism $A_{\tau} \simeq A_{\tau'}$. On the underlying real vector space $M_{2\times 2}(\KR)$, this map is given by multiplication by $\gamma$ from the right. The form $E$ is invariant under this action, i.e. $E(m_1\gamma, m_2\gamma) = E(m_1, m_2)$ for any $m_1, m_2 \in M_{2\times 2}(\KR)$. The isomorphism therefore is an isomorphism of polarized abelian varieties $(A_{\tau}, H_{\tau}) \simeq (A_{\tau'}, H_{\tau'})$.

\subsection{The modular family $\pi: M \lra C$} Let $\Gamma \subset \order_{B, +}^*$ be of finite index and torsion free. The set of matrices
\begin{equation} \label{gl}
    \gamma_{\lambda} := \left(\begin{array}{cc}
            id_{2\times 2} & \lambda \\
             0_{2\times 2} & \gamma
        \end{array}\right) \in Sl_4(\KR) \quad \lambda \in \order_B, \gamma \in \Gamma
\end{equation}
defines a group $\Gamma_{\order_B}$. Think of $\KC^2\times \Sieg_1$ as embedded into the standard chart $\{X_3=1\}$ of $\PN_3$. $\Gamma_{\order_B}$ acts projectively on $\KC^2\times \Sieg_1$:
  \[\gamma_{\lambda}(z, \tau) = \big(\frac{z + \lambda_{\tau}}{c\tau+d}, \frac{a\tau+b}{c\tau+d}\big), \quad \lambda \in \order_B, \gamma=\left(\begin{array}{cc}
                                                    a & b \\
                                                    c & d
                                                 \end{array}\right) \in \Gamma\]
The action is free and properly discontinously; the quotient $M$ is a smooth threefold carrying a holomorphic projective structure.

The homomorphism $\Gamma_{\order_B} \lra \Gamma, \gamma_{\lambda} \lra \gamma$
gives a holomorphic map $\pi: M \lra C=\Gamma/\Sieg_1$. The fiber over $[\tau]$ is isomorphic to $A_{\tau}$. Both $M$ and $C$ are compact and the $H_\tau$ glue to a $\pi$--ample $H_{\pi}$ on $M$. The map $\Sieg_1 \lra \KC^2 \times \Sieg_1, \tau \mapsto (\tau, 0)$ induces a section $C \lra M$ of $\pi$. In other words, $M$ is an abelian scheme.

There is an inclusion
  \[\order_B \hookrightarrow End(M),\]
i.e., any $\lambda \in \order_B$ induces a $C$--endomorphism of $M$.

We call $\pi: M \lra C$ a {\em modular family of fake elliptic curves}. For our purposes, $M$ depends on a choice of $B$, $B \hookrightarrow M_{2\times 2}(\KR)$, $\order_B$ and $\Gamma$ as above.

\subsection{Factors of automorphy}
The group $\Gamma$ is torsion free. Then $a(\tau, \gamma) = (c\tau + d) \in H^1(\Gamma, \O_{\Sieg_1}^*)$ is a well defined factor of automorphy defining a theta characteristic on $C$ which we denote by $\frac{K_C}{2}$ for simplicity. 

The holomorphic tangent bundle $T_M$ is defined by the following factor in $H^1(\Gamma_{\order_B}, Gl_3(\O_{\KC^2\times \Sieg_1}))$:
 \begin{equation} \label{foa}
    a(\gamma_{\lambda}, (z,\tau)) := 
\mbox{$\frac{1}{c\tau + d}$}\left(\begin{array}{cc}
id_{2\times 2} & \frac{{}_1\lambda - c(z + \lambda_{\tau})}{c\tau + d} \\
0_{1 \times 2} & \frac{1}{c\tau + d}\end{array}
 \right),
  \end{equation}
%  \begin{equation} \label{foa}
%    a(\gamma_{\lambda}, (z=(z_1, z_2),\tau)) := (c\tau + d)\left(\begin{array}{cc}
%id_{2\times 2} & 0_{2 \times 1}\\
%\frac{{}_1\lambda^t - c(z^t + \lambda_{\tau}^t)}{c\tau + d} & \frac{1}{c\tau + d}\end{array}
% \right)^{-1}.
%  \end{equation}
where we introduce the following notation:
  \[{}_1m \;\; \mbox{denotes the first column of } m \in M_{2\times 2}(\KR).\]
The block form of $a$ corresponds to the exact sequence
  \[0 \lra T_{M/C} \simeq \pi^*\big(\mbox{$\frac{-K_C}{2} \oplus \frac{-K_C}{2}$}\big) \lra T_M \lra \pi^*T_C \lra 0.\]
Taking determinants we find in particular
 \begin{equation} \label{KMKC}
   K_M = 2\pi^*K_C.
 \end{equation}

\section{Splitting submanifolds}
\setcounter{equation}{0}

From Proposition~\ref{split} we get the following possibilities: 
\begin{lemma} \label{SplitCurves}
  Let $\pi: M \lra C$ be a modular family of fake elliptic curves as in \S 3. Let $N$ be a compact submanifold that splits in $M$. Then either
  \begin{enumerate}
  \item $N$ is a fiber of $\pi$ or
  \item $N$ is a smooth elliptic curve in a fiber of $\pi$ or
  \item $N$ is an {\'e}tale multisection of $\pi$.
  \end{enumerate}
\end{lemma}
In Lemma~\ref{fibnonspl} below we show that, as in the case of modular families of elliptic curves, a fiber in fact never splits.

\begin{proof}
  For the proof we may assume $g_C > 1$. We first show that $N$ cannot be a surface: 

{\em 1. Case. } Let $N$ be a complex surface that splits in $M$. By proposition~\ref{split} $N$ carries a holomorphic projective connection. By \cite{KO}, $N \simeq \PN_2(\KC)$ or $N$ is a finite {\'e}tale quotient of an abelian surface or $N$ is a ball quotient. As $M$ does not contain a rational curve, $N \not\simeq \PN_2(\KC)$. 

If $N$ is not a fiber of $\pi$, then $\pi_N=\pi|_N: N \lra C$ is surjective. By proposition~\ref{split}, $K_N \equiv \frac{3}{4}K_M|_N$. Since $K_M$ is trivial on every fiber of $\pi$ by (\ref{KMKC}), $K_N$ is (numerically) trivial on every fiber of $\pi_N$. The adjunction formula shows that $\pi_N$ defines an elliptic fibration on $N$.

A ball quotient is hyperbolic and cannot contain an elliptic curve (any holomorphic map $\KC \lra \B_2(\KC)$ is constant). A torus does not have a surjective map to a curve of genus $>1$ for the same reason (any holomorphic map $\KC^2 \lra \B_1(\KC)$ is constant). Therefore, $N$ must be a fiber of $\pi$.

{\em 2. Case.} Let $N$ be a compact Riemann surface that splits in $M$. Then $g_N > 0$, as $M$ does not contain rational curves. By Proposition~\ref{split}, $K_N \equiv \frac{1}{2}K_M|_N$.

If $N$ is contained in a fiber of $\pi$, then $\deg K_N = \frac{1}{2}K_M.N = 0$ and $N$ is an elliptic curve. If $\pi_N: N \lra C$ is surjective of degree $d$, then
  \[K_N \sim \pi_N^*K_C + R\]
by Hurwitz formula, where $R$ is effective and $R = 0$ iff $\pi_N$ is {\'e}tale. By (\ref{KMKC}), $K_M \equiv 2\pi^*K_C$. Then Hurwitz' formula reads
  \[K_N \equiv \frac{1}{2}K_M|_N + R.\]
Combining this with $K_N \equiv \frac{1}{2}K_M|_N$ from above we find $R = 0$ and $\pi_N$ {\'e}tale.
\end{proof}

\section{Fibers are non--split}
\setcounter{equation}{0}

Fix some $\tau \in \Sieg_1$. Then $A_\tau$ from \Formel{FEC} can be viewed as the fiber of $\pi$ over $[\tau] \in \Sieg_1/\Gamma$. Denote by $\iota: A_\tau \lra M$ the inclusion map. Then we have
  \begin{equation} \label{AinM}
     0 \lra N_{A_\tau/M}^* \simeq \O_{A_\tau} \lra \iota^*\Omega_M^1 \lra \Omega^1_{A_\tau} \simeq \O_{A_{\tau}}^{\oplus 2} \lra 0.
  \end{equation}
The following lemma seems to be well-known within the theory of Kuga fiber spaces. Since we could not find a reference, we give a proof.

\begin{lemma} \label{fibnonspl}
  The induced map $H^0(A_\tau, N^*_{A_\tau/M}) \hookrightarrow H^0(A_\tau, \iota^*\Omega_M^1)$ is an isomorphism. In particular, $A_\tau$ is non--split in $M$.
\end{lemma}
Before the proof we first illustrate the general method. We have (see section~3 for notations)
 \[\iota_*: \pi_1(A_\tau) \simeq \order_B \lra \pi_1(M) \simeq \Gamma_{\order_B}, \quad \iota_*(\lambda) = id_{\lambda}.\]
The bundle $i^*\Omega_M^1$ is given by the pull back to $A_\tau$ of the factor of automorphy dual to \Formel{foa} (i.e., the pull back of $(a(\gamma_{\lambda}, (z,\tau))^{-1})^t$). We get the following representation of the fundamental group
  \begin{equation} \label{rhoA}
\rho_A: \pi_1(A_\tau) \simeq \order_B \lra Gl_3(\KC), \quad \lambda \mapsto \left(\begin{array}{cc}
         id_{2\times 2} & 0_{2\times 1} \\
           -{}_1\lambda^t & 1
     \end{array}\right).
 \end{equation}
The bundle $\iota^*\Omega_M^1$ is a flat bundle and \Formel{AinM} is a sequence of flat bundles. In terms of representations, $\KC^3$ as a $\pi_1(A_{\tau})$ module is the extension of the trivial one dimensional and the trivial two dimensional module.

\begin{rem}
  Let $0 \lra K \lra V \lra Q \lra 0$ be a short exact sequence of complex vector spaces which are $G$ modules, where $G = \pi_1(G)$ for some complex manifold $M$. We get a short exact sequence of flat vector bundles $0 \lra {\mathcal K} \lra {\mathcal V} \lra {\mathcal Q} \lra 0$.

If $V \simeq K \oplus Q$ as $G$--modules, then the sequence of flat vector bundles splits holomorphically. The convers, however, need not be true, the map
  \[H^1(G, Q^* \otimes K) \lra H^1(M, {\mathcal Q}^*\otimes {\mathcal K})\]
of obstruction spaces might have a non trivial kernel.
\end{rem}

\begin{proof} (of Lemma~\ref{fibnonspl})
The space $H^0(A_\tau, \iota^*\Omega_M^1)$ is isomorphic to the space of holomorphic $f = (f_1, f_2, f_3) \in \O_{A_\tau}^{\oplus 3}$ satisfying ($z = (z_1, z_2)^t$)
  \begin{equation} \label{invcond}
f(z+\lambda_\tau) = \rho_A(\lambda)f(z) \quad \mbox{for any } \lambda \in \order_B.
  \end{equation}
The sections corresponding to $H^0(A_\tau, N^*_{A_\tau/M}) \hookrightarrow H^0(A_\tau, \iota^*\Omega_M^1)$ are $f = (0,0,b)^t$, $b \in \KC$ constant. We have to show that these are all.

Let $f$ satisfy \Formel{invcond}. Then $f_i(z+\lambda_\tau) = f_i(z)$ for $i = 1,2$ implies $f_1, f_2$ constant. The third equation reads $f_3(z + \lambda_\tau) = -\lambda_{11}f_1 - \lambda_{21}f_2 + f_3(z)$, $\lambda \in \order_B$ arbitrary. Then $\partial f_3/\partial z_i$ are constant and hence $f_3(z_1, z_2) = a_1z_1 + a_2z_2 + b$, $a_1, a_2, b \in \KC$. After subtracting $(0,0,b)^t$ we may assume $b = 0$. 

The third equation now reduces to 
  \[0 = \lambda_{11}f_1 + \lambda_{21}f_2 + (a_1, a_2)\lambda_\tau \quad \mbox{for any } \lambda \in \order_B.\]
The four generators $\lambda_1, \lambda_2, \lambda_2, \lambda_4$ of $\order_B$ yield four linear equations in $(f_1, f_2, a_1, a_2)$. The defining matrix in $M_{4 \times 4}(\KC)$ is
   \[\left(\begin{array}{cccc}
            {}_1\lambda_1 &  {}_1\lambda_2 & {}_1\lambda_3 & {}_1\lambda_4 \\
             \lambda_{1,\tau} & \lambda_{2,\tau} & \lambda_{3, \tau} & \lambda_{4, \tau}
    \end{array}\right) = \left(\begin{array}{cc}
 \id_{2\times 2} & 0\\
 \tau \cdot \id_{2\times 2} & \id_{2 \times 2}
\end{array}\right)
\cdot 
\left(\begin{array}{cccc}
{}_1\lambda_1 &  {}_1\lambda_2 & {}_1\lambda_3 & {}_1\lambda_4 \\
   {}_2\lambda_1 &  {}_2\lambda_2 & {}_2\lambda_3 & {}_2\lambda_4 
  \end{array}\right).\]
The determinant is nonzero since $\lambda_1, \lambda_2, \lambda_2, \lambda_4$ are $\KR$--independent. The only solution therefore is $(f_1, f_2, a_1, a_2) = (0,0,0,0)$.
\end{proof}

\section{Elliptic curves in fake elliptic curves}
\setcounter{equation}{0}

Let $\pi: M \lra C$ be a modular family of fake elliptic curves. We call matrices in the kernel of
 \[B^{\times} \lra PGl_2(\KC)\]
projectively trivial. Under this map, any $b \in B^{\times}$ acts on $\PN_1(\KC)$. Identify 
  \[\Sieg_1 = \{[\tau:1] \mid \tau \in \Sieg_1\}.\]
If some $b \in B^{\times}$ fixes some $\tau \in \Sieg_1$, then $\det b > 0$ and $b$ fixes $\Sieg_1$.

The next proposition shows the existence of elliptic curves in special fibers of $\pi$:

\begin{proposition}
  The abelian surface $A_{\tau}$ from \Formel{FEC} is isogeneous to a product of elliptic curves if and only if $\tau$ is a fixed point of some projectively non--trivial $b \in B^{\times}$. 

For a given elliptic curve $E$ the following conditions are equivalent:
\begin{enumerate}
  \item There exists a non constant holomorphic map $f: E \lra A_{\tau}$. 
  \item $A_{\tau}$ is isogeneous to $E \times E$.
  \item ${\rm End}_{\KQ}(A_{\tau}) \simeq M_{2\times 2}({\rm End}_{\KQ}(E))$.
\end{enumerate}
In any case $E$ and $A_\tau$ are CM.
\end{proposition}

For an elliptic curve $E$, $End_{\KQ}(E) \simeq \KQ$ or $End_{\KQ}(E)$ is an imaginary quadratic extension of $\KQ$. A given elliptic curve $E$ appears in a fiber of the modular family if and only if $End_{\KQ}(E)$ is a splitting field of $B$.

\begin{proof} $A_{\tau}$ is isogeneous to a product of elliptic curves if and only if there exists a non constant homomorphism  $\varphi: E_{\tau'} \lra A_{\tau}$ for some $\tau' \in \Sieg_1$, where $E_{\tau'} = \KC/\Lambda_{\tau'}$ for $\Lambda_{\tau'} = \KZ\tau' + \KZ$. Such a homomorphism exists if and only if we find $0 \not= \lambda, \mu \in \order_B$ such that
\begin{equation} \label{FundRel} 
  \tau'\lambda\left(\begin{array}{c}
                  \tau \\ 1 \end{array}\right) = \mu\left(\begin{array}{c}
                  \tau \\ 1 \end{array}\right).
 \end{equation}
Indeed, a given $\varphi$ may be interpreted as a $\KC$-linear map $\varphi: \KC \lra \KC^2$ satisfying $\varphi(\KZ\tau'+\KZ) \subset \order_{B, \tau}$ and then $\lambda$ and $\mu$ are induced by $\varphi(1)$ and $\varphi(\tau')$, respectively. Conversely, given $\mu, \lambda$ as above, $\varphi$ is defined by $\KC \lra \KC^2, z \mapsto z\lambda_{\tau}$. The map $\varphi$ is non--constant if and only if $\lambda \not= 0$.

The existence of a non constant $\varphi$ then implies $\tau$ is a fixed point of $b := \lambda^{-1}\mu \in B^{\times}$. If conversely $\tau$ is a fixed point of some projectively non trivial $b \in B^{\times}$, we immediately get \Formel{FundRel}.

The equivalence of the four points is not difficult to prove. We skip details since it is not important for our purposes. We only recall that $End_{\KQ}(A_\tau)$ always contains $B$. 
\end{proof}

\begin{rem} \label{simpleform}
  Let $\varphi: E_{\tau'} \lra A_{\tau}$ be as above induced by $0 \not= \lambda, \mu \in \order_B$ such that \Formel{FundRel} holds. Modulo isogenies we may assume $\lambda = id_{2\times 2}$.

Indeed, choose $n \in \KN$ such that $\mu' := n\lambda^{-1}\mu \in \order_B$, consider $\tilde{\varphi}: E_{n\tau'} \lra A_\tau$ induced by $z \mapsto (\tau, 1)^tz$. Then $\lambda \circ \tilde{\varphi} = \varphi \circ p$ where $p: E_{n\tau'} \lra E_{\tau'}$ is the canonical map and $\lambda \in End(A_\tau)$. 

Note that $\lambda$ also induces an isogeny of $M$.
\end{rem}

\

\section{Elliptic curves in fibers split}
\setcounter{equation}{0}

Let $\iota: E \hookrightarrow A_\tau$ be an elliptic curve in the fiber $A_\tau$ of $\pi$. Then we have
\begin{equation} \label{EinM}
   0 \lra N_{E/M}^* \lra \iota^*\Omega^1_M \lra \Omega_E^1 = \O_E \lra 0
 \end{equation}
Our aim is to prove that an elliptic curve in a fiber of $\pi: M \lra C$ splits in $M$. Choose $\tau' \in \Sieg_1$ and
  \[\varphi: E_{\tau'} \lra A_{\tau}\]
such that $E = \varphi(E_{\tau'})$. Then \Formel{EinM} splits holomorphically if and only if $\varphi^*$ of \Formel{EinM} splits. By Remark~\ref{simpleform} we may assume that $\varphi$ as a map $\KC \lra \KC^2$ is given by $z \mapsto (\tau, 1)^tz$, sending $\tau'$ to 
  \begin{equation} \label{EVmu}
     \tau'\left(\begin{array}{c}
                 \tau \\
                   1
               \end{array}\right) = \mu_\tau
  \end{equation}
for some $0 \not= \mu \in \order_B$. The matrix $\mu$ remains fixed for the rest of this section.

The idea is the same as in \S 4: we will view \Formel{EinM} as a sequence of flat bundles coming from representations of $\pi_1(E_{\tau'})$ and compute dimensions of spaces of global sections. We have
  \[\varphi_*: \pi_1(E_{\tau'}) = \KZ\tau' + \KZ \lra \pi_1(M) \simeq \Gamma_{\order_B}, \quad \varphi_*(m\tau' + n) = id_{m\mu + n id_{2\times 2}}\]
in the notation \Formel{gl}. The flat bundle $\varphi^*\Omega_M^1$ is given by the pull back of the factor of automorphy dual to \Formel{foa}. We get the representation
\[\rho_E: \pi_1(E_{\tau'}) \lra Gl_3(\KC), \quad \rho(m\tau' + n) = \left(\begin{array}{ccc}
             1 & 0 & 0 \\
             0 & 1 & 0 \\
             -m\mu_{11}-n & -m\mu_{21} & 1
          \end{array}\right)\]
as in (\ref{rhoA}). It turns $\KC^3$ into a $\pi_1(E_{\tau'})$--module. The bundle $\Omega_{E_{\tau'}}^1$ is given by the trivial one dimensional module, and $d\varphi$ is induced by the $\pi_1(E_{\tau'})$--map
 \[\KC^3 \lra \KC, \quad (z_1, z_2, z_3)^t \lra (z_1\tau' + z_2).\]
The kernel is the $\pi_1(E_{\tau'})$--module corresponding to $\varphi^*N^*_{E/M}$. Note that there is an inclusion map
  \begin{equation} \label{inclN}
        \varphi^*N^*_{A_\tau/M} \simeq \O_{E_{\tau'}} \hookrightarrow \varphi^*N^*_{E/M}
  \end{equation}
into the subbundle $\varphi^*N^*_{E/M}$ of $\varphi^*\Omega_M^1$.

\begin{proposition} \label{ellinfsplit} Let $\varphi: E_{\tau'} \to A_{\tau}$ be an elliptic curve in the fiber $A_{\tau}$ of $\pi$. Then $\dim_{\KC}H^0(E_{\tau'}, \varphi^*\Omega_M^1) = 2$ and the global differential map $D\varphi: H^0(E_{\tau'}, \varphi^*\Omega_M^1) \lra H^0(E_{\tau'},  \Omega_{E_{\tau'}}^1)$ is surjective. In particular, elliptic curves in fibers of $\pi$ split in $M$.
\end{proposition}

\begin{proof}
  The space $H^0(E_{\tau'}, \varphi^*\Omega_M^1)$ is isomorphic to the space of holomorphic $f = (f_1, f_2, f_3) \in \O_{\KC}^3$ satisfying
 \begin{equation} \label{invcondE}
f(z+m\tau' + n) = \rho_E(m\tau' + n)f(z) \quad \mbox{for any } m, n \in \KZ.
  \end{equation}
The sections coming from $\varphi^*N^*_{A_\tau/M} \simeq \O_{E_{\tau'}} \hookrightarrow \varphi^*N^*_{E/M}$ correspond to $f = (0,0,b)^t$, $b \in \KC$ constant. We have to show that there is an additional dimension.

Let $f = (f_1, f_2, f_3)$ satisfy \Formel{invcondE}. As in the proof of lemma~\ref{fibnonspl} $f_1, f_2$ are constant while $f_3(z) = az+b$ for some $a, b \in \KC$. We may assume $b = 0$. Then \Formel{invcondE} reduces to
  \[am\tau' + an = -(m\mu_{11}+n)f_1 -m\mu_{21}f_2 \quad \mbox{for any } m, n \in \KZ.\]
After eliminating $a$ from the equations obtained for $(m,n) = (1,0)$ and $(0,1)$ we get 
\begin{equation} \label{fin}
   \tau'f_1 = \mu_{11}f_1 + \mu_{21}f_2.
\end{equation}
Conversely, if $f_1, f_2$ satisfies \Formel{fin}, then  $f = (f_1, f_2, -f_1z)$ is a solution to \Formel{invcondE}. By \Formel{EVmu}, $\tau'$ is an eigenvalue of $\mu$ and therefore also of $\mu^t$. A nonzero vector $(f_1, f_2)$ satisfies \Formel{fin} if and only if $(f_1, f_2)^t$ is an eigenvector of $\mu$ to the eigenvalue $\tau'$. Therefore $h^0(E_{\tau'}, \varphi^*\Omega_M^1) = 2$.

The eigenspace of $\mu^t$ to $\tau'$ is
  \[\langle \left(\begin{array}{c} 1 \\ \-\bar{\tau} \end{array}\right)\rangle_{\KC}.\]
Since $d\varphi(1, -\bar{\tau}, -z) = \tau - \bar{\tau} = 2\Im m\tau \not= 0$, $D\varphi$ is surjective.
\end{proof}

\begin{proof}[Proof of Proposition~\ref{result}.] Lemma~\ref{SplitCurves} and \ref{fibnonspl} show that a compact submanifold $N$ that splits in $M$ is an {\'e}tale multisection of $\pi$ or an elliptic curve in a fiber of $\pi$.

{\'E}tale multisections $\tilde{C}$ split because of
  \[0 \lra \pi^*K_C \lra \Omega_M^1 \lra \Omega_{M/C}^1\lra 0.\]
The restriction to $\tilde{C}$ shows $d\pi|_{\tilde{C}}$ splits $\Omega_M^1|_{\tilde{C}} \lra K_{\tilde{C}}$. Elliptic curves in fibers split in $M$ by Proposition~\ref{ellinfsplit}.
\end{proof}

\end{document}